\documentclass[12pt, reqno]{amsart}
\usepackage{amsmath, amsthm, amscd, amsfonts, amssymb, graphicx, color}
\usepackage[bookmarksnumbered, colorlinks, plainpages]{hyperref}
\hypersetup{colorlinks=true,linkcolor=red, anchorcolor=green, citecolor=cyan, urlcolor=red, filecolor=magenta, pdftoolbar=true}

\newtheorem{theorem}{Theorem}[section]
\newtheorem{lemma}[theorem]{Lemma}
\newtheorem{proposition}[theorem]{Proposition}
\newtheorem{corollary}[theorem]{Corollary}
\theoremstyle{definition}

\theoremstyle{remark}
\newtheorem{remark}[theorem]{Remark}
\numberwithin{equation}{section}

\allowdisplaybreaks

\renewcommand{\leq}{\leqslant}
\renewcommand{\geq}{\geqslant}
\renewcommand{\epsilon}{\varepsilon}

\begin{document}
\sloppy

\setcounter{page}{1}

\title[]{A new complemented subspace for the Lorentz sequence spaces, with an application to its lattice of closed ideals}

\author{Ben Wallis}
\address{
Division of Math/Science/Business,
Kishwaukee College,
Malta, IL 60150, United States
}
\email{\textcolor[rgb]{0.00,0.00,0.84}{benwallis@live.com}}

\subjclass[2010]{Primary 46B20; Secondary 46B45, 47L20.}

\keywords{Lorentz sequence spaces, complemented subspaces, lattice of closed ideals.}


\begin{abstract}
We show that every Lorentz sequence space $d(\textbf{w},p)$ admits a 1-complemented subspace $Y$ distinct from $\ell_p$ and containing no isomorph of $d(\textbf{w},p)$.  In the general case, this is only the second nontrivial complemented subspace in $d(\textbf{w},p)$ yet known.  We also give an explicit representation of $Y$ in the special case $\textbf{w}=(n^{-\theta})_{n=1}^\infty$ ($0<\theta<1$) as the $\ell_p$-sum of finite-dimensional copies of $d(\textbf{w},p)$.  As an application, we find a sixth distinct element in the lattice of closed ideals of $\mathcal{L}(d(\textbf{w},p))$, of which only five were previously known in the general case.
\end{abstract}

\maketitle

\section{Introduction}

Little is known about the complemented subspace structure of Lorentz sequence spaces $d(\textbf{w},p)$.  Until recently, the only nontrivial complemented subspace discussed in the literature was $\ell_p$ (\cite{ACL73}).  Then, in \cite{Wa20} it was shown that for certain weights $\textbf{w}$ (see Theorem \ref{NUC-sum} below), the space $d(\textbf{w},p)$ contains a 1-complemented subspace isomorphic to $(\bigoplus_{n=1}^\infty\ell_\infty^n)_p$.  Up to now, these were the only nontrivial complemented subspaces known to exist.

In this short note we show that each Lorentz sequence space admits a 1-complemented subspace $Y$ distinct from $\ell_p$ (\S2).  We also give an explicit representation of $Y$ for the case $\textbf{w}=(n^{-\theta})_{n=1}^\infty$ ($0<\theta<1$), as the $\ell_p$-sum of finite-dimensional copies of $d(\textbf{w},p)$ (\S3).  Finally, as an application we find a sixth distinct element in the lattice of closed ideals in the operator algebra $\mathcal{L}(d(\textbf{w},p))$, where only five were previously known in the general case (\S4).

Let's set up the main notation we need to use.  We begin by fixing $\mathbb{K}\in\{\mathbb{R},\mathbb{C}\}$.  Denote by $\Pi$ the set of all permutations of $\mathbb{N}$, and denote by $\mathbb{W}$ the set of all sequences $\textbf{w}=(w_n)_{n=1}^\infty\in c_0\setminus\ell_1$ satisfying
$$
1=w_1\geq w_2\geq w_3\geq\cdots>0.
$$
Fix $1\leq p<\infty$ and $\textbf{w}\in\mathbb{W}$.  For each $(a_n)_{n=1}^\infty\in\mathbb{K}^\mathbb{N}$ we set
$$
\left\|(a_n)_{n=1}^\infty\right\|_{d(\textbf{w},p)}:=\sup_{\pi\in\Pi}\left(\sum_{n=1}^\infty|a_{\pi(n)}|^pw_n\right)^{1/p},
$$
and let $d(\textbf{w},p)$ denote the linear space of all $(a_n)_{n=1}^\infty\in\mathbb{K}^\mathbb{N}$ with $\|(a_n)_{n=1}^\infty\|_{d(\textbf{w},p)}<\infty$ endowed with the norm $\|\cdot\|_{d(\textbf{w},p)}$, called a {\bf Lorentz sequence space}.  Recall that if $(a_n)_{n=1}^\infty\in c_0$ then there exists a ``decreasing rearrangement'' $(\hat{a}_n)_{n=1}^\infty$ of $(|a_n|)_{n=1}^\infty$.  In this case, the rearrangement inequality gives us
$$
\left\|(a_n)_{n=1}^\infty\right\|_{d(\textbf{w},p)}=\left(\sum_{n=1}^\infty\hat{a}_n^pw_n\right)^{1/p}\;\;\;\text{ for all }(a_n)_{n=1}^\infty\in c_0.
$$
Since $d(\textbf{w},p)\subset c_0$ as linear spaces (although not as normed spaces), this represents an alternative formulation of the Lorentz sequence space norm.

For each $i,k\in\mathbb{N}$ we define
$$
W_k:=\sum_{n=1}^kw_n\;\;\text{ and }\;\;w_i^{(k)}:=\frac{1}{W_k}\sum_{n=(i-1)k+1}^{ik}w_n
$$
and $\textbf{w}^{(k)}:=(w_i^{(k)})_{i=1}^\infty$.  It is readily apparent that $\textbf{w}^{(k)}\in\mathbb{W}$.  When $p$ is clear from context, we also set
$$
d_i^{(k)}:=\frac{1}{W_k^{1/p}}\sum_{n=(i-1)k+1}^{ik}d_n.
$$
It's routine to verify that $(d_i^{(k)})_{i=1}^\infty$ is a normalized basic sequence isometrically equivalent to the $d(\textbf{w}^{(k)},p)$ basis.  If necessary, we may sometimes abuse this notation; for instance, if $(j_k)_{k=1}^\infty$ is a sequence in $\mathbb{N}$, then we could write $((d_i^{(j_k)})_{i=1}^k)_{k=1}^\infty$ for appropriately-translated successive normalized constant-coefficient blocks of lengths $j_k$.

Our main tool for finding complemented subspaces of $d(\textbf{w},p)$ is the fact that every constant-coefficient block basic sequence of a symmetric basis spans a 1-complemented subspace (cf.\ e.g.\ \cite[Proposition 3.a.4]{LT77}).  We will use this well-known fact freely and without further reference.

\section{Lorentz sequence spaces contain at least two nontrivial complemented subspaces}

The first discovery of a nontrivial complemented subspace in $d(\textbf{w},p)$ came almost half a century ago, with the following result.

\begin{theorem}[{\cite[Lemma 1]{ACL73}}]\label{ACL-subsequence-lemma}
Fix $1\leq p<\infty$ and $\textbf{\emph{w}}\in\mathbb{W}$, and let
$$
x_i=\sum_{n=p_i}^{p_{i+1}-1}a_nd_n,\;\;\;i\in\mathbb{N},
$$
form a seminormalized block basic sequence in $d(\textbf{\emph{w}},p)$. If $a_n\to 0$ then $(x_i)_{i=1}^\infty$ admits a subsequence equivalent to $\ell_p$ and complemented in $d(\textbf{\emph{w}},p)$.
\end{theorem}

\noindent By taking sufficiently long constant-coefficient blocks, it follows that $d(\textbf{w},p)$ contains a 1-complemented copy of $\ell_p$.  Much later was shown the following.

\begin{theorem}[{\cite[Theorem 4.3]{Wa20}}]\label{NUC-sum}
Let $1\leq p<\infty$ and $\textbf{\emph{w}}=(w_n)_{n=1}^\infty\in\mathbb{W}$.  If
$$
\inf_{k\in\mathbb{N}}\frac{\sum_{n=1}^{2k}w_n}{\sum_{n=1}^kw_n}=1
$$
then $d(\textbf{\emph{w}},p)$ admits a 1-complemented subspace spanned by constant-coefficient blocks and isomorphic to $(\bigoplus_{n=1}^\infty\ell_\infty^n)_p$.
\end{theorem}

\noindent Thanks in large part to the ideas of William B.\ Johnson, our main result in this section is to generalize this to all Lorentz sequence spaces, as follows.

\begin{theorem}\label{nontrivial-complemented}
Let $1\leq p<\infty$ and $\textbf{\emph{w}}\in\mathbb{W}$.  Then there exists an increasing sequence $(N_k)_{k=1}^\infty\in\mathbb{N}^\mathbb{N}$ such that $((d_i^{(k)})_{i=1}^{N_k})_{k=1}^\infty$ spans a 1-complemented subspace $Y$ which contains no isomorph of $d(\textbf{\emph{w}},p)$ and which is not isomorphic to $\ell_p$.
\end{theorem}

To prove it, we need a few preliminaries.

\begin{lemma}\label{subsymmetric-complement}
If $1<p<\infty$ then every complemented subspace of $L_p[0,1]$ with a subsymmetric basis $(x_n)_{n=1}^\infty$ is isomorphic to either $\ell_p$ or $\ell_2$.
\end{lemma}

\begin{proof}
The case $p=2$ is trivial since every complemented subspace of $L_2[0,1]$ is isomorphic to $\ell_2$.  For the case $p>2$ recall from \cite[Corollary 6]{KP62} that every seminormalized basic sequence in $L_p[0,1]$, $p\in(2,\infty)$, admits a subsequence equivalent to $\ell_p$ or $\ell_2$, and so since $(x_n)_{n=1}^\infty$ is also subsymmetric then it is in fact equivalent to $\ell_p$ or $\ell_2$.  In case $1<p<2$, since $(x_n)_{n=1}^\infty$ is complemented in $L_p[0,1]$, its corresponding sequence of biorthogonal functionals $(x_n^*)_{n=1}^\infty$ is contained in $L_{p'}[0,1]$, where $\frac{1}{p}+\frac{1}{p'}=1$.  Since $p'>2$, a subsequence of $(x_n^*)_{n=1}^\infty$ is equivalent to $\ell_{p'}$ or $\ell_2$, whence by subsymmetry $(x_n)_{n=1}^\infty$ is equivalent to $\ell_p$ or $\ell_2$.
\end{proof}

\begin{lemma}\label{canonical-copy-complemented}
Let $X$ be a Banach space whose canonical isometric copy in $X^{**}$ is complemented.  Then for any free ultrafilter $\mathcal{U}$ on $\mathbb{N}$, the canonical copy of $X$ in $X^\mathcal{U}$ is complemented in $X^\mathcal{U}$.
\end{lemma}

\begin{proof}
Let $q:X\to X^{**}$ denote the canonical embedding, and define the norm-1 linear operator $V:\ell_\infty(X)\to X^{**}$ by the rule
$$
V(x_n)_{n=1}^\infty=\underset{\mathcal{U}}{\text{weak*-lim}}\,qx_n,
$$
which exists by the weak*-compactness of $B_{X^{**}}$ together with the fact that if $K$ is a compact Hausdorff space then for each $(k_n)_{n=1}^\infty\in K^\mathbb{N}$ the (unique) limit $\lim_\mathcal{U}k_n$ exists in $K$.  Note that if $\lim_\mathcal{U}x_n=0$ then $V(x_n)_{n=1}^\infty=0$ and so $V$ induces an operator $\widehat{V}:X^\mathcal{U}\to X^{**}$ which agrees with $V$ along the diagonal.  In particular, $\widehat{V}$ sends the canonical copy of $X$ in $X^\mathcal{U}$ isomorphically to the canonical copy of $X$ in $X^{**}$.
\end{proof}

\begin{theorem}\label{complemented-in-Lp}
Fix $1\leq p<\infty$, and let $(x_n)_{n=1}^\infty$ be a basis for a Banach space $X$ whose canonical copy in $X^{**}$ is complemented.  If the finite-dimensional spaces $[x_n]_{n=1}^N$, $N\in\mathbb{N}$, are uniformly complemented in $L_p(\mu)$ for some measure $\mu$, then $X$ is complemented in $L_p[0,1]$.
\end{theorem}

\begin{proof}
Let $X_N=[x_n]_{n=1}^N$ and $P_N:X\to X_N$ the projection onto $X_N$.  By uniform complementedness of $X_N$, we can find uniformly bounded linear operators $A_N:X_N\to L_p(\mu)$ and $B_N:L_p(\mu)\to X_N$ such that $B_NA_N$ is the identity on $X_N$.  Let $\mathcal{U}$ be any free ultrafilter on $\mathbb{N}$.  Define the bounded linear operators $A:X\to L_p(\mu)^\mathcal{U}$ by the rule $Ax=(A_NP_Nx)_\mathcal{U}$, and $B:L_p(\mu)^\mathcal{U}\to X^\mathcal{U}$ by $B(y_N)_\mathcal{U}=(B_Ny_N)_\mathcal{U}$.    Let $x\in\text{span}(x_n)_{n=1}^\infty$ so that, for some $k\in\mathbb{N}$,
$$
BAx=(P_1x,\cdots,P_kx,x,x,\cdots)_\mathcal{U}=x^\mathcal{U}.
$$
By continuity, $BA$ is the canonical injection of $X$ into $X^\mathcal{U}$.  Since this is complemented by Lemma \ref{canonical-copy-complemented}, we have the identity on $X$ factoring through $L_p(\mu)^\mathcal{U}$.

It was proved in \cite[Theorem 3.3]{He80} that ultrapowers preserve $L_p$ lattice structure, and in particular $L_p(\mu)^\mathcal{U}$ is isomorphic to $L_p(\nu)$ for some measure $\nu$.  Although $L_p(\nu)$ itself is nonseparable, we could pass to the closed sublattice generated by $AX$ to find a space isomorphic to a separable $L_p$ containing a complemented copy of $X$.  Due mostly to a famous result of Lacey and Wojtaszczyk, it's known that separable and infinite-dimensional $L_p$ spaces are isomorphic to either $\ell_p$ or $L_p[0,1]$ (\cite[\S4, p15]{JL01}).  This means an isomorph of $X$ is complemented in $L_p[0,1]$.
\end{proof}

An immediate corollary to Lemma \ref{subsymmetric-complement} and Theorem \ref{complemented-in-Lp} is as follows.

\begin{corollary}\label{no-uniformly-complemented-copies}
Let $1<p<\infty$ and $\textbf{\emph{w}}\in\mathbb{W}$.  Then no $L_p(\mu)$ space contains uniformly complemented copies of $[d_n]_{n=1}^N$, $N\in\mathbb{N}$.
\end{corollary}

Now we're ready to prove the main result of this section.

\begin{proof}[Proof of Theorem \ref{nontrivial-complemented}]
Fix $k\in\mathbb{N}$, and note that $(d_i^{(k)})_{k=1}^\infty$ is isometric to the $d(\textbf{w}^{(k)},p)$ basis.  Consider the case where $p=1$.  Then we can choose the $N_k$'s large enough that each $(d_i^{(k)})_{i=1}^{N_k}$ fails to be $k$-equivalent to $\ell_1^{N_k}$, and hence $((d_i^{(k)})_{i=1}^{N_k})_{k=1}^\infty$ fails to be equivalent to $\ell_1$.  As $\ell_1$ has a unique unconditional basis by a result of Lindenstrauss and Pe\l{}czy\'{n}ski, it follows that $Y$ is not isomorphic to $\ell_1$.

Next, consider the case where $1<p<\infty$.  By Corollary \ref{no-uniformly-complemented-copies} we can select $N_k$'s large enough that $[d_i^{(k)}]_{i=1}^{N_k}$ fails to be $k$-complemented in $\ell_p$.  As $[d_i^{(k)}]_{i=1}^{N_k}$'s are all 1-complemented in $Y$, that means $Y$ is not isomorphic to $\ell_p$.

It remains to show that $Y$ contains no isomorph of $d(\textbf{w},p)$.  Suppose towards a contradiction that it does.  As $(d_n)_{n=1}^\infty$ is weakly null (cf.\ e.g.\ \cite[Proposition 1]{ACL73}) we can use the gliding hump method together with symmetry to find a normalized block sequence of $((d_i^{(k)})_{i=1}^{N_k})_{k=1}^\infty$ equivalent to $(d_n)_{n=1}^\infty$.  However, every such block sequence is also a block sequence w.r.t.\ $(d_n)_{n=1}^\infty$ with coefficients tending to zero.  By Theorem \ref{ACL-subsequence-lemma} it follows that $(d_n)_{n=1}^\infty$ admits a subsequence equivalent to $\ell_p$, which is impossible.
\end{proof}

\section{A special case}

In this section we show that when $\textbf{w}=(n^{-\theta})_{n=1}^\infty$ for some fixed $0<\theta<1$, the space $Y$ described in Theorem \ref{nontrivial-complemented} can be chosen to be isomorphic to the space
$$
Y_{\textbf{w},p}:=\left(\bigoplus_{N=1}^\infty D_N\right)_p,
$$
where $D_N:=[d_n]_{n=1}^N$ for each $N\in\mathbb{N}$.  As usual, we require some preliminaries.

\begin{lemma}\label{integral-estimate}
Let $0<\theta<1$ and $j,k\in\mathbb{N}$.  Then
$$
((j+1)/k+1)^{1-\theta}-((j+1)/k)^{1-\theta}
\leq\frac{\sum_{n=j+1}^{j+k}n^{-\theta}}{\sum_{n=1}^kn^{-\theta}}
\leq\frac{(j/k+1)^{1-\theta}-(j/k)^{1-\theta}}{2^{1-\theta}-1}.
$$
\end{lemma}

\begin{proof}
Observe that the map
$$
f(t)=(1+1/t)^{1-\theta}-(1/t)^{1-\theta}
$$
is increasing on $[1,\infty)$, and hence has a minimum $f(1)=2^{1-\theta}-1$.  Hence,
\begin{align*}
((j+1)/k+1)^{1-\theta}-((j+1)/k)^{1-\theta}
&\leq\frac{(j+k+1)^{1-\theta}-(j+1)^{1-\theta}}{k^{1-\theta}-\theta}
\\&=\frac{\int_{j+1}^{j+k+1}t^{-\theta}\;dt}{1+\int_1^kt^{-\theta}\;dt}
\\&\leq\frac{\sum_{n=j+1}^{j+k}n^{-\theta}}{\sum_{n=1}^kn^{-\theta}}
\\&\leq\frac{\int_j^{j+k}t^{-\theta}\;dt}{\int_1^{k+1}t^{-\theta}\;dt}
\\&=\frac{(j+k)^{1-\theta}-j^{1-\theta}}{(k+1)^{1-\theta}-1}
\\&=\frac{(j/k+1)^{1-\theta}-(j/k)^{1-\theta}}{(1+1/k)^{1-\theta}-(1/k)^{1-\theta}}
\\&\leq\frac{(j/k+1)^{1-\theta}-(j/k)^{1-\theta}}{2^{1-\theta}-1}.
\end{align*}
\end{proof}

\begin{lemma}\label{uniformly-equivalent-weights}
Let $0<\theta<1$ and $\textbf{\emph{w}}=(w_n)_{n=1}^\infty=(n^{-\theta})_{n=1}^\infty\in\mathbb{W}$.  Then
$$
\frac{1-\theta}{2}\cdot w_i
\leq w_i^{(k)}
\leq\frac{2-2^\theta}{2^{1-\theta}-1}\cdot w_i
\;\;\;\text{ for all }i,k\in\mathbb{N}.
$$
In particular, if $1\leq p<\infty$ then there is a constant $C\in[1,\infty)$, depending only on $\theta$, such that
$$
(d_n)_{n=1}^\infty\approx_C(d_i^{(k)})_{i=1}^\infty\;\;\;\text{ for all }k\in\mathbb{N}.
$$
\end{lemma}

\begin{proof}
We can assume $i,k\geq 2$.  Observe that
$$
t\mapsto t-(t-1)^{1-\theta}\cdot t^\theta
$$
is decreasing on $[2,\infty)$ and hence has the maximum $2-2^\theta$.  Also, the function
$$
t\mapsto t-(t-1/2)^{1-\theta}\cdot t^\theta
$$
is decreasing on $[2,\infty)$ and hence has infimum
$$
\lim_{t\to\infty}\left(t-(t-1/2)^{1-\theta}\cdot t^\theta\right)=\frac{1-\theta}{2}.
$$
Thus, by the above together with Lemma \ref{integral-estimate},
\begin{align*}
\frac{1-\theta}{2}\cdot i^{-\theta}
&\leq\left(i-(i-1/2)^{1-\theta}\cdot i^\theta\right)i^{-\theta}
\\&=i^{1-\theta}-(i-1/2)^{1/\theta}
\\&\leq(i+1/k)^{1-\theta}-(i-1+1/k)^{1/\theta}
\\&\leq\frac{\sum_{n=(i-1)k+1}^{ik}n^{-\theta}}{\sum_{n=1}^kn^{-\theta}}
\\&\leq\frac{i^{1-\theta}-(i-1)^{1-\theta}}{2^{1-\theta}-1}
\\&=\frac{i-(i-1)^{1-\theta}\cdot i^\theta}{2^{1-\theta}-1}\cdot i^{-\theta}
\\&\leq\frac{2-2^\theta}{2^{1-\theta}-1}\cdot i^{-\theta}.
\end{align*}
\end{proof}

\begin{remark}\label{disjoint-vectors}
Suppose $x=\sum_{n\in A}a_nd_n$ and $y=\sum_{n\in B}b_nd_n$ for finite and disjoint sets $A,B\subset\mathbb{N}$, where $(a_n)_{n\in A}$ and $(b_n)_{n\in B}$ are sequences of scalars.  Then
$$
\|x+y\|^p\leq\|x\|^p+\|y\|^p.
$$
\end{remark}

\begin{lemma}\label{last-step}
Let $(j_k)_{k=1}^\infty$ be a sequence of positive integers, and for each $k$ set
$$J_k=j_1+2j_2+3j_3+\cdots+kj_k.$$
Suppose that there are constants $A,B\in(0,\infty)$ such that
\begin{equation}\label{A-constant}
w_i^{(j_k)}\leq Aw_i
\end{equation}
and
\begin{equation}\label{B-constant}
Bw_i\leq\frac{1}{W_{j_k}}\sum_{n=J_{k-1}+(i-1)j_k+1}^{J_{k-1}+ij_k}w_n
\end{equation}
for all $i=1,\cdots,k$ and all $k\in\mathbb{N}$.  Then $((d_i^{(j_k)})_{i=1}^k)_{k=1}^\infty$ is equivalent to the canonical $Y_{\textbf{\emph w},p}$ basis.
\end{lemma}

\begin{proof}
Due to \eqref{A-constant} we have $(d_i^{(j_k)})_{i=1}^k\lesssim_A d(\textbf{w},p)^k$.  Now, using Remark \ref{disjoint-vectors}, for any finitely-supported scalar sequence $((a_i^{(k)})_{i=1}^k)_{k=1}^\infty$,
\begin{align*}
\left\|\sum_{k=1}^\infty\sum_{i=1}^k a_i^{(k)}d_i^{(j_k)}\right\|^p
&\leq\sum_{k=1}^\infty\left\|\sum_{i=1}^ka_i^{(k)}d_i^{(j_k)}\right\|^p
\\&\leq A^p\sum_{k=1}^\infty\left\|(a_i^{(k)})_{i=1}^k\right\|_{d(\textbf{w},p)}^p
\\&=A^p\left\|((a_i^{(k)})_{i=1}^k)_{k=1}^\infty\right\|_{Y_{\textbf{w},p}}^p.
\end{align*}
For the reverse inequality, let $(\hat{a}_i^{(k)})_{i=1}^k$ denote the decreasing rearrangement of $(|a_i^{(k)}|)_{i=1}^k$.  Then, applying \eqref{B-constant},
\begin{align*}
\left\|\sum_{k=1}^\infty\sum_{i=1}^k a_i^{(k)}d_i^{(j_k)}\right\|^p
&=\left\|\sum_{k=1}^\infty\sum_{i=1}^k\frac{a_i^{(k)}}{W_{j_k}^{1/p}}\sum_{n=J_{k-1}+(i-1)j_k+1}^{J_{k-1}+ij_k}d_n\right\|^p
\\&\geq\sum_{k=1}^\infty\sum_{i=1}^k\frac{\hat{a}_i^{(k)p}}{W_{j_k}}\sum_{n=J_{k-1}+(i-1)j_k+1}^{J_{k-1}+ij_k}w_n
\\&\geq B\sum_{k=1}^\infty\sum_{i=1}^k\hat{a}_i^{(k)p}w_i
\\&=B\left\|((a_i^{(k)})_{i=1}^k)_{k=1}^\infty\right\|_{Y_{\textbf{w},p}}^p.
\end{align*}
\end{proof}

\begin{theorem}\label{main-equivalent}
Let $(j_k)_{k=1}^\infty$ and $(J_k)_{k=1}^\infty$ be is as in Lemma \ref{last-step}.  Suppose there is $M\in[1,\infty)$ such that
$$
\frac{J_{k-1}}{j_k}\leq M,\;\;\;\text{ for all }k=2,3,4,\cdots.
$$
Then $((d_i^{(j_k)})_{i=1}^k)_{k=1}^\infty$ is equivalent to the canonical $Y_{\textbf{\emph w},p}$ basis.
\end{theorem}

\begin{proof}
Due to Lemma \ref{last-step}, it suffices to show that \eqref{B-constant} and \eqref{A-constant} both hold.  To do this, fix an arbitrary $k\in\mathbb{N}$.  We may assume without loss of generality that $j_k\geq 2$.  Now, by Lemma \ref{integral-estimate},
\begin{align*}
\frac{1}{W_{j_k}}\sum_{n=J_{k-1}+(i-1)j_k+1}^{J_{k-1}+ij_k}w_n
&\geq\left(\frac{J_{k-1}+(i-1)j_k+1}{j_k}+1\right)^{1-\theta}-\left(\frac{J_{k-1}+(i-1)j_k+1}{j_k}\right)^{1-\theta}
\\&=\left(\frac{J_{k-1}}{j_k}+i+\frac{1}{j_k}\right)^{1-\theta}-\left(\frac{J_{k-1}}{j_k}+i-1+\frac{1}{j_k}\right)^{1-\theta}
\\&\geq\left(\frac{J_{k-1}}{j_k}+i\right)^{1-\theta}-\left(\frac{J_{k-1}}{j_k}+i-1+\frac{1}{2}\right)^{1-\theta}
\\&=i^\theta\left[\left(\frac{J_{k-1}}{j_k}+i\right)^{1-\theta}-\left(\frac{J_{k-1}}{j_k}+i-\frac{1}{2}\right)^{1-\theta}\right]w_i
\end{align*}
Applying the Mean Value Theorem to the function $x\mapsto(\phi+x)^{1-\theta}$, $\phi\in[1,\infty)$, we can find $x_\phi\in(-1/2,0)$ such that
$$
\phi^{1-\theta}-(\phi-1/2)^{1-\theta}=\frac{(1-\theta)(\phi+x_\phi)^{-\theta}}{2}
\geq\frac{(1-\theta)\phi^{-\theta}}{2}.
$$
Hence, letting $\phi=J_{k-1}/j_k+i$, we have
\begin{align*}
i^\theta\left[\left(\frac{J_{k-1}}{j_k}+i\right)^{1-\theta}-\left(\frac{J_{k-1}}{j_k}+i-\frac{1}{2}\right)^{1-\theta}\right]
&\geq i^\theta\left[\frac{(1-\theta)(J_{k-1}/j_k+i)^{-\theta}}{2}\right]
\\&=\frac{1-\theta}{2}\left(\frac{i}{J_{k-1}/j_k+i}\right)^\theta
\\&\geq\frac{1-\theta}{2}\left(\frac{1}{M+1}\right)^\theta
\end{align*}
This proves \eqref{B-constant}, and \eqref{A-constant} follows immediately from Lemma \ref{uniformly-equivalent-weights}.
\end{proof}

Taking inductively $j_1=1$ and $j_{k+1}=J_k$, the following is now immediate.

\begin{corollary}
Let $1\leq p<\infty$, $0<\theta<1$, and $\textbf{\emph{w}}=(w_n)_{n=1}^\infty=(n^{-\theta})_{n=1}^\infty\in\mathbb{W}$.  Then $d(\textbf{\emph{w}},p)$ admits a 1-complemented subspace isomorphic to $Y_{\textbf{\emph w},p}$.
\end{corollary}

\section{application to the lattice of closed ideals}

In \cite{KPSTT12} was shown (among other results) that the lattice of closed ideals for the operator algebra $\mathcal{L}(d(\textbf{w},p))$ can be put into a chain:
$$
\{0\}
\subsetneq\mathcal{K}(d(\textbf{w},p))
\subsetneq\mathcal{SS}(d(\textbf{w},p))
\subsetneq\mathcal{S}_{d(\textbf{w},p)}(d(\textbf{w},p))
\subsetneq\mathcal{L}(d(\textbf{w},p)).
$$
Here, $\mathcal{K}$ denotes the compact operators, $\mathcal{SS}$ the strictly singular operators, and $\mathcal{S}_{d(\textbf{w},p)}$ the ideal of operators which fail to be bounded below on any isomorph of $d(\textbf{w},p)$.  While in \cite[Corollary 2.7]{Wa20}, for the special case where $1<p<2$ and $\textbf{w}\in\mathbb{W}\cap\ell_{2/(2-p)}$, a chain of distinct closed ideals with cardinality of the continuum were identified lying between $\mathcal{K}(d(\textbf{w},p))$ and $\mathcal{SS}(d(\textbf{w},p)$, for the general case, the only distinct elements known were those of the above chain.

For an operator $T$, let $\mathcal{J}_T$ denote the class of operators factoring through $T$.  If $Z$ is any Banach space, we then set $\mathcal{J}_Z=\mathcal{J}_{Id_Z}$.  By Theorem \ref{new-ideal} below, we can extend the chain above as follows:
\begin{multline*}
\{0\}
\subsetneq\mathcal{K}(d(\textbf{w},p))
\subsetneq\mathcal{SS}(d(\textbf{w},p))
\subsetneq(\overline{\mathcal{J}_{\ell_p}}\vee\mathcal{SS})(d(\textbf{w},p))
\\\subsetneq\mathcal{S}_{d(\textbf{w},p)}(d(\textbf{w},p))
\subsetneq\mathcal{L}(d(\textbf{w},p)).
$$
\end{multline*}
Furthermore, by \cite[Corollary 3.2 and Theorem 5.3]{KPSTT12} together with the fact that $d(\textbf{w},p)$ has the approximation property, any additional distinct closed ideals in $\mathcal{L}(d(\textbf{w},p))$ must lie between $\mathcal{K}(d(\textbf{w},p))
$ and $\mathcal{SS}(d(\textbf{w},p))$, or else between $(\overline{\mathcal{J}_{\ell_p}}\vee\mathcal{SS})(d(\textbf{w},p))$ and $\mathcal{S}_{d(\textbf{w},p)}(d(\textbf{w},p))$.

To prove Theorem \ref{new-ideal}, we need a couple of preliminary results.

\begin{proposition}\label{proper-ideal}
Let $X$ and $Z$ be an infinite-dimensional Banach spaces such that $Z^2\approx Z$, and $X$ fails to be isomorphic to a complemented subspace of $Z$.  Then $\overline{\mathcal{J}_Z}(X)$ is a proper ideal in $\mathcal{L}(X)$.  Furthermore, if $P\in\mathcal{L}(X)$ is a projection with image isomorphic to $Z$, then 
$$\mathcal{J}_P(X)=\mathcal{J}_Z(X).$$
\end{proposition}

\begin{proof}
Since $Z^2\approx Z$, \cite[Lemma 2.2]{KPSTT12} guarantees that $\mathcal{J}_Z(X)$ is an ideal in $\mathcal{L}(X)$.  Suppose towards a contradiction that that $Id_X\in\mathcal{J}_Z(X)$.  Then $Id_X=AB$ for operators $A\in\mathcal{L}(Z,X)$ and $B\in\mathcal{L}(X,Z)$.  By \cite[Lemma 2.1]{KPSTT12}, $BX$ is complemented in $Z$ and isomorphic to $X$, which contradicts our hypotheses.  It follows that $\mathcal{J}_Z(X)$ is a proper ideal in $\mathcal{L}(X)$.  Recall that the closure of a proper ideal in a unital Banach algebra is again proper; in particular, $\overline{\mathcal{J}_Z}(X)$ is a proper ideal in $\mathcal{L}(X)$.

To prove the ``furthermore'' part, assume $A\in\mathcal{L}(Z,X)$ and $B\in\mathcal{L}(X,Z)$. Let $Q:Z\to X$ be the canonical embedding so that $PQ=Id_Z$ and hence $AB=APQB\in\mathcal{J}_P(X)$.  It follows that $\mathcal{J}_Z(X)\subseteq\mathcal{J}_P(X)$, and the reverse inclusion is even more obvious.
\end{proof}

For the next result, $\mathcal{F}$ denotes the class of finite-rank operators and $\mathcal{E}$ the class of inessential operators.  Recall also that a basis $\mathcal{B}$ is called {\it semispreading} whenever every subsequence of $\mathcal{B}$ is dominated by $\mathcal{B}$ itself.  In particular, the unit vector basis of $\ell_p$ is semispreading.

\begin{proposition}[{\cite[Corollary 3.8]{LLR04}}]\label{factorable-ops}
Let $Z$ be a Banach space with a semi-spreading basis $(z_n)$, and let $X$ be a Banach space with basis $(x_n)$ such that any seminormalized block sequence of $(x_n)$ contains a subsequence equivalent to $(z_n)$ and spanning a complemented subspace of $X$.  Then
$$
\{0\}\subsetneq\overline{\mathcal{F}}(X)=\mathcal{K}(X)=\mathcal{SS}(X)=\mathcal{E}(X)\subsetneq\overline{\mathcal{J}_Z}(X),
$$
and any additional distinct closed ideals must lie between $\overline{\mathcal{J}_Z}(X)$ and $\mathcal{L}(X)$.
\end{proposition}

In the proof of what follows, we use the fact that if $\mathcal{I}$ and $\mathcal{J}$ are ideals in $\mathcal{L}(X)$, then $\overline{\mathcal{I}}\vee\overline{\mathcal{J}}=\overline{\mathcal{I}+\mathcal{J}}$.

\begin{theorem}\label{new-ideal}
Fix $1\leq p<\infty$ and $\textbf{\emph{w}}\in\mathbb{W}$.  Let $Y$ be as in Theorem \ref{nontrivial-complemented}, and $P_Y\in\mathcal{L}(d(\textbf{\emph{w}},p))$ any continuous linear projection onto $Y$. Then $$P_Y\in\mathcal{S}_{d(\textbf{\emph{w}},p)}(d(\textbf{\emph{w}},p))\setminus(\overline{\mathcal{J}_{\ell_p}}\vee\mathcal{SS})(d(\textbf{\emph{w}},p)).$$
\end{theorem}

\begin{proof}
Let $P_{\ell_p}\in\mathcal{L}(d(\textbf{w},p))$ be any projection onto an  isomorophic copy of $\ell_p$ spanned by basis vectors of $Y$.  (Such a copy exists by Theorem \ref{ACL-subsequence-lemma}.)  By Theorem \ref{nontrivial-complemented}, $Y$ contains no isomorph of $d(\textbf{w},p)$ and hence $P_Y\in\mathcal{S}_{d(\textbf{w},p)}(d(\textbf{w},p))$.  Since $\mathcal{S}_{d(\textbf{w},p)}(d(\textbf{w},p))$ is the unique maximal ideal in $\mathcal{L}(d(\textbf{w},p))$, and $\mathcal{J}_{P_{\ell_p}}(d(\textbf{w},p))=\mathcal{J}_{\ell_p}(d(\textbf{w},p))$ by Proposition \ref{proper-ideal}, it's sufficient to prove that
$P_Y\notin(\overline{\mathcal{J}_{P_{\ell_p}}}\vee\mathcal{SS})(d(\textbf{w},p)).$

Next, we claim that  $P_Y\in(\overline{\mathcal{J}_{P_{\ell_p}}}\vee\mathcal{SS})(d(\textbf{w},p))$ only if $Id_Y\in(\overline{\mathcal{J}_{\ell_p}}\vee\mathcal{SS})(Y)$.  To prove it, fix $\epsilon>0$, and suppose there are $A,B\in\mathcal{L}(d(\textbf{w},p))$ and $S\in\mathcal{SS}(d(\textbf{w},p))$ such that
$$
\left\|AP_{\ell_p}B+S-P_Y\right\|<\epsilon.
$$
Let $J_Y:Y\to d(\textbf{w},p)$ be an embedding satisfying $P_YJ_Y=J_Y$, or $P_YJ_Y=Id_Y$ when viewed as an operator in $\mathcal{L}(Y)$.  Composing $P_Y$ on the left and $J_Y$ on the right, we have
$$
\left\|P_YAP_{\ell_p}BJ_Y+P_YSJ_Y-Id_Y\right\|_{\mathcal{L}(Y)}<\|P_Y\|\cdot\epsilon\cdot\|J_Y\|.
$$
On the other hand, since $AP_{\ell_p}=A|_YP_{\ell_p}$ and $P_{\ell_p}=P_{\ell_p}P_Y$, we have
$$
P_YAP_{\ell_p}BJ_Y
=(P_YA|_Y)P_{\ell_p}(P_YBJ_Y)
$$
and hence
$$
\left\|(P_YA|_Y)P_{\ell_p}(P_YBJ_Y)+P_YSJ_Y-Id_Y\right\|_{\mathcal{L}(Y)}<\|P_Y\|\cdot\epsilon\cdot\|J_Y\|.
$$
Since $\mathcal{J}_{\ell_p}(Y)=\mathcal{J}_{P_{\ell_p}}(Y)$ by Proposition \ref{proper-ideal}, where $P_{\ell_p}$ is likewise viewed as an operator in $\mathcal{L}(Y)$, from the above together with the ideal property of $\mathcal{SS}$, the claim follows.

Let $\mathcal{B}_Y=((d_i^{(k)})_{i=1}^{N_k})_{k=1}^\infty$ denote the canonical basis of $Y$ from Theorem \ref{nontrivial-complemented}.  Note that since $\mathcal{B}_Y$ is made up of constant coefficient blocks of $(d_n)$ of increasing length, any seminormalized blocks of $\mathcal{B}_Y$ will contain a subsequence equivalent to $\ell_p$  by Theorem \ref{ACL-subsequence-lemma}.  In fact, in \cite[Lemma 15]{CL74} this result was refined to show that we can choose that subsequence to span a complemented subspace of $d(\textbf{w},p)$, and hence of $Y$ itself.  We can therefore apply Theorem \ref{factorable-ops} to conclude that $\mathcal{SS}(Y)\subset\overline{\mathcal{J}_{\ell_p}}(Y)$.  Meanwhile, again by Proposition \ref{proper-ideal}, $\overline{\mathcal{J}_{\ell_p}}(Y)$ is a proper ideal in $\mathcal{L}(Y)$, which means $Id_Y\notin\overline{\mathcal{J}_{\ell_p}}(Y)$.  Hence, $P_Y\notin(\overline{\mathcal{J}_{P_{\ell_p}}}\vee\mathcal{SS})(d(\textbf{w},p))$ as desired.
\end{proof}

\end{document}